\documentclass[11pt,  a4paper]{amsart}
\usepackage{eucal}
\usepackage{graphicx}
\usepackage[english]{babel}
\usepackage{amsmath,amsthm,amssymb,amsfonts,amscd,amsbsy,amsxtra}
\usepackage{epsfig,color}
\usepackage{subfigure}
\usepackage{psfrag}
\usepackage[arrow,matrix, curve]{xy}
\usepackage{amsmath}
\usepackage{enumerate}
\usepackage[T1]{fontenc} 
\usepackage[utf8]{inputenc}
\usepackage{mathrsfs}
\usepackage[shortlabels]{enumitem}
\usepackage{mathtools}
\usepackage[bookmarks=true,pdfstartview=FitH, pdfborder={0 0 0}, colorlinks=true,citecolor=red, linkcolor=blue]{hyperref}
\usepackage{stmaryrd}
\usepackage{bbm}
\usepackage{aliascnt}
\usepackage{a4wide}
\usepackage{nicefrac}
\usepackage{array}
\newcommand {\R}	{\mathbb{R}}
\newcommand {\N}	{\mathbb{N}}

\DeclareMathOperator{\Id}{Id}

\newcommand{\loc}{{loc}}
\newcommand{\dX}{{\partial X}}

\newcommand{\sL}{\mathcal{L}}

\newcommand{\rC}{\mathrm{C}}
\newcommand{\rW}{\mathrm{W}}
\newcommand{\rL}{\mathrm{L}}

\newcommand{\ddn}{\frac{\partial^a}{\partial \nu^g}}

\renewcommand{\epsilon}{\varepsilon}
\renewcommand{\div}{\text{div}}

\advance\textheight by 2cm

\setlist[enumerate]{font = \normalfont}

\bibliography{book.bib}
\theoremstyle{plain}
\newtheorem{thm}{Theorem}[section]
\newaliascnt{cor}{thm}
\newaliascnt{prop}{thm}
\newaliascnt{lem}{thm}
\newtheorem{cor}[cor]{Corollary}

\newtheorem{lem}[lem]{Lemma}
\aliascntresetthe{cor}
\aliascntresetthe{prop}
\aliascntresetthe{lem} 
\newcounter{stp}
\newcounter{stpi}
\newcounter{stpci}
\newcounter{stpiii}

 \theoremstyle{theorem}

\theoremstyle{definition}
\newaliascnt{defn}{thm}
\newaliascnt{asu}{thm}
\newaliascnt{con}{thm}
\aliascntresetthe{defn}
\aliascntresetthe{asu}
\aliascntresetthe{con}
\theoremstyle{remark}
\newaliascnt{rem}{thm}
\newaliascnt{exa}{thm}
\newaliascnt{masu}{thm}
\newaliascnt{nota}{thm}
\newaliascnt{sett}{thm}
\newtheorem{rem}[rem]{Remark}

\newtheorem{nota}[nota]{Notation}
\newtheorem{sett}[sett]{Abstract Setting}
\aliascntresetthe{rem}
\aliascntresetthe{exa}
\aliascntresetthe{masu}
\aliascntresetthe{nota}
\aliascntresetthe{sett}
\numberwithin{equation}{section}

\title [Elliptic operators on continuous functions on manifolds with boundary]
{Strictly elliptic operators with generalized Wentzell boundary conditions on continuous functions on manifolds with boundary}
\author{Tim Binz}
\subjclass{47D06, 34G10, 47E05, 47F05}%
\keywords{Wentzell boundary conditions, Dirichlet-to-Neumann operator, analytic semigroup, Riemmanian manifolds}%

\date{\today}%

\begin{document}

\maketitle

\begin{abstract}
We prove that strictly elliptic operators with generalized Wentzell boundary conditions generate analytic semigroups of angle $\nicefrac{\pi}{2}$ on the space of continuous function on a compact manifold with boundary.
\end{abstract}

\section{Introduction}

We start from a strictly elliptic differential operators $A_m$ with domain $D(A_m)$ on the space $C(\overline{M})$ of continuous functions on a smooth, compact, orientable Riemannian manifold $(\overline{M},g)$ with smooth boundary $\partial M$. Moreover, let $C$ be a strictly elliptic differential operator on the boundary, take $\ddn:D(\ddn)\subset C(\overline{ M}) \to C(\partial {M})$ to be the outer conormal derivative, and functions $\eta, \gamma\in C(\partial M)$ with $\eta$ strictly positive
and a constant $q > 0$. In this setting we define the operator $A^B \subset A_m$ with \emph{generalized Wentzell boundary conditions} by requiring
\begin{equation}\label{eq:bc-W-Lap} 
f\in D(A^B)
\quad:\iff\quad f \in D(A_m) \cap D(B), \ 
A_m f\big|_{\partial M}= q \cdot C f|_{\partial M} - \eta \cdot \ddn f+\gamma\cdot f\big|_{\partial\Omega}.
\end{equation}
On a bounded domain $\Omega \subset \R^n$ with sufficiently smooth boundary $\partial \Omega$, Favini, Goldstein, Goldstein, Obrecht and Romanelli in \cite{FGGR:10} showed that for $A_m = \Delta_\Omega$ and $C = \Delta_{\partial \Omega}$ the operator 
$A^B$ generates an analytic semigroup of angle $\nicefrac{\pi}{2}$ on 
$C(\overline{\Omega})$. In a preprint Goldstein, Goldstein and Pierre in \cite{GGP:17} generalized this statement to arbitrary elliptic differential operators of the form
$A_m f := \sum_{l,k = 1}^n \partial_l (a^{kl} \partial_k f)$ and $C \varphi := \sum_{l,k = 1}^n \partial_l (\alpha^{kl} \partial_k \varphi )$. 

\smallskip 

Our main theorem \autoref{mainthm} generalizes these results to arbitrary strictly elliptic operators
$A_m$ and $C$ on smooth, compact, orientable Riemannian manifolds with smooth boundary.

\smallskip

The situation $q = 0$ on bounded, smooth domains in $\R^n$ was studied by Engel and Fragnelli \cite{EF:05} and, on smooth, compact, orientable Riemannian manifolds by \cite{Bin:18b}. 

\smallskip

The paper is organized as follows. In the second section we introduce the abstract setting from \cite{EF:05} and \cite{BE:18} for our problem. In the third section 
we study the special case that $A_m$ is the Laplace-Beltrami operator and $B$ is the normal derivative. In the last section we 
generalize to arbitrary strictly elliptic operators and their conormal derivatives.

\smallskip

Throughout the whole paper we use the Einstein notation for sums and write
$x_i y^i$ shortly for $\sum_{i=1}^{n}x_i y^i$. Moreover we denote by $\hookrightarrow$ a continuous and by $\stackrel{c}{\hookrightarrow}$ a compact embedding.

\section{The abstract setting}

As in \cite[Section 2]{EF:05} the basis of our investigation is the following

\begin{sett}\label{set:AS}
Consider
\begin{enumerate}[(i)]
\item two Banach spaces $X$ and $\dX$, called \emph{state} and 
\emph{boundary space}, respectively;
\item a densely defined \emph{maximal operator}
$A_m \colon D(A_m) \subset X \rightarrow X$;
\item a \emph{boundary (or trace) operator} $L \in \sL(X,\dX)$;
\item a \emph{feedback operator} $B \colon D(B) \subseteq X \rightarrow \dX$.
\end{enumerate}
\end{sett}

Using these spaces and operators we define the operator $A^B:D(A^B)\subset X\to X$
with abstract \emph{generalized Wentzell boundary conditions} as
\begin{equation}\label{eq:W-BC}
A^B f := A_m f, \quad 
D(A^B):= \bigl\{ f \in D(A_m) \cap D(B) : LA_mf = Bf \bigr\} .  
\end{equation}
%If $B=0$, the boundary conditions defined by \eqref{eq:W-BC} are called \emph{pure Wentzell boundary conditions}. 
For an interpretation of Wentzell- as ``dynamic boundary conditions'' we refer to \cite[Sect.~2]{EF:05}.

\smallskip
In the sequel we need the following operators. 

\begin{nota} 
The kernel of $L$ is a closed subspace and we
consider the restriction $A_0\subset A_m$ given by
\begin{alignat*}{3}
A_{0}:D(A_0)\subset X\to X,\quad D(A_{0}) := \{ f \in D(A_m) : Lf = 0 \}. 
\end{alignat*}
\end{nota}

The \emph{abstract
Dirichlet operator associated with $A_m$}
is, if it exists, 
\begin{equation*}
L^{A_m}_0 := (L|_{\ker(A_m)})^{-1} \colon \dX
\rightarrow \ker(A_m) \subseteq X,
\end{equation*}
i.e. $L^{A_m}_0 \varphi = f$ is the unique solution of the abstract Dirichlet problem
\begin{equation}
\begin{cases}
A_m f = 0, \\
Lf = \varphi .
\end{cases}\label{Dirichlet Problem}
\end{equation}
If it is clear which operator $A_m$ is meant, we simply write $L_0$.

%\begin{masu}\label{masu}
%	In the following we assume that
%	\begin{enumerate}[(i)]
%	\item	The operator $A_0$ is a weak Hille-Yosida operator on $X$;
%	\item 	the operator $B \colon D(B) \subset X \to \partial X$ is relatively $A_0$-bounded of bound $0$;
%	\item 	the Dirichlet operator $L_0 \in \mathcal{L}(\dX,X)$ exists.
%	\end{enumerate}
%\end{masu}

\smallskip 

Finally, we introduce the \emph{abstract Dirichlet-to-Neumann operator associated with $(A_m,B)$}, defined by 
\begin{equation*}
N^{A_m, B} \varphi :=BL^{A_m}_0 \varphi,
\quad
D(N^{A_m, B}) := \bigl\{\varphi \in \dX : L^{A_m}_0 \varphi \in D(B) \bigr\}.
\end{equation*}
If it is clear which operators $A_m$ and $B$ are meant, we write $N=N^{A_m,B}$ and call it the (abstract) Dirichlet-to-Neumann operator.

\section{Laplace-Beltrami operator with generalized Wentzell boundary conditions}

Take now as maximal operator $A_m \colon D({A}_m) \subset \rC(\overline{M}) \to \rC(\overline{M})$ the Laplace-Beltrami operator $\Delta_M^g$ with domain $D(A_m) := \left\{ f \in \bigcap_{p > 1} \rW^{2,p}_{\loc}(M) \cap \rC(\overline{M}) \colon A_m f \in \rC(\overline{M}) \right\}$.
Moreover consider another strictly elliptic differential operator $C \colon D(C) \subset \rC(\partial M) \to \rC(\partial M)$ in divergence form on the boundary space. 
To this end, take real valued functions 
\begin{align*}
\alpha^k_j = \alpha_j^k \in \rC^\infty(\partial M), \quad \beta_j \in \rC(\partial M), \quad \gamma \in \rC(\partial M), \quad 1\leq j,k\leq n,
\end{align*}
such that $\alpha_j^k$ are strictly elliptic,~i.e.
\begin{align*}
\alpha_j^k(q) g^{jl}(q) X_k (q) X_l(q) > 0
\end{align*}
for all co-vectorfields $X_k,X_l$ on $\partial M$ with $(X_1(q),\dots,X_n(q)) \not = (0,\dots,0)$. Let $\alpha = (\alpha_j^k)_{j,k = 1,\dots,n}$ the $1$-$1$-tensorfield and $\beta = (\beta_j)_{j = 1,\dots,n}$. Moreover we denote by $|\alpha|$ the determinate of $\alpha$ and define $C \colon D(C) \subset \rC(\partial M) \to \rC(\partial M)$ by
\begin{align}
C \varphi &:= \sqrt{|\alpha|} \div_g \left( \frac{1}{\sqrt{|\alpha|}} \alpha \nabla_{\partial M}^g \varphi \right) + \langle \beta, \nabla_{\partial M}^g \varphi \rangle + \gamma \cdot \varphi, \label{Def: C} \\
D(C) &:= \left\{ \varphi \in \bigcap_{p > 1} \rW^{2,p}(\partial M) \colon C \varphi \in \rC(\partial M) \right\}. \notag
\end{align}
In order to define the feedback operator we first consider
$B_0 \colon D(B_0) \subset \rC(\overline{M}) \to \rC(\partial M)$ given by 
\begin{equation*}
B_0 f := - g( a \nabla_M^g f, \nu_g), \quad D(B_0) := \left\{ f \in \bigcap_{p > 1} \rW^{2,p}_{\loc}(M) \cap \rC(\overline{M}) \colon B_0 f \in \rC(\partial M) \right\}.
\end{equation*}
This leads to the feedback operator $B \colon D(B) \subset \rC(\overline{M}) \to \rC(\partial M)$ is defined as
\begin{align*}
Bf &:= q \cdot C L f - \eta \cdot g(\nabla_M^g f, \nu_g), \\
D(B) &:= \{ f \in D(A_m) \cap D(B_0) \colon Lf \in D(C) \},
\end{align*}
where $L \colon \rC(\overline{M}) \to \rC(\partial M) \colon f\mapsto f|_{\partial M}$ denotes the trace operator and $q > 0$ and $\eta \in \rC(\overline{M})$ is positive. 
Now we consider the operator with Wentzell boundary conditions on $\rC(\overline{M})$ as defined in \eqref{eq:W-BC} with respect to the operators $A_m$ and $B$ above.

\smallskip  

Note that the feedback operator $B$ can be splitted into
\begin{equation*}
B = q \cdot C L + \eta \cdot B_0.
\end{equation*}

The following proof is inspired by \cite{Eng:03} and similar to \cite[Ex.~5.3]{BE:18}.

%\smallskip 

%First one notes that $B$ is relatively $A_0$-bounded.

\begin{lem}\label{GWBC B A_0}
	The operator $B$ is relatively $A_0$-bounded of bound $0$.
\end{lem}
\begin{proof}
	Since $D(A_0) \subset \ker(L)$, the operators $B$ and $\eta \cdot B_0$ coincide on $D(A_0)$. Hence it remains to prove the statement for the operator $B_0$. 
	By \cite[Chap.~5., Thm.~1.3]{Tay:96} 
	and the closed graph theorem we obtain
	\begin{align*}
	[D(A_0)] \hookrightarrow \rW^{2,p}(M) .
	\end{align*}
	Rellich's embedding (see \cite[Thm.~6.2, Part III.]{Ada:75}) implies
	\begin{equation*}
	\rW^{2,p}(M) \stackrel{c}{\hookrightarrow}\rC^{1,\alpha}(M) \stackrel{c}{\hookrightarrow}\rC^1(\overline{M})
	\end{equation*}
	for $p > \frac{m-1}{1-\alpha}$, so we obtain
	\begin{align*}
	[D(A_0)] \stackrel{c}{\hookrightarrow}\rC^1(\overline{M}) \hookrightarrow \rC(\overline{M}) . 
	\end{align*}
	Therefore, by Ehrling's lemma (cf. \cite[Thm.~6.99]{RR:04}), for every $\varepsilon >0$ there exists a constant $C_\varepsilon >0$ such that
	\begin{equation*}
	\| f \|_{\rC^1(\overline{M})} \leq \varepsilon \| f \|_{A_0} + C_\varepsilon \| f \|_X
	\end{equation*}
	for every $f \in D(A_0)$. Since $B_0 \in \mathcal{L}(\rC^1(\overline{M}),\dX)$,
	this implies the claim. 
\end{proof}

\begin{lem}
	The operator $N^{\Delta_m,B_0}$ is relatively $C$-bounded of bound $0$. 
\end{lem}
\begin{proof}
	Let $W := -(\Delta_{\partial M}^g)^{\nicefrac{1}{2}}$ and remark that by the proof of \cite[Thm.~3.8]{Bin:18b} there exists a relatively $W$-bounded perturbation $P$ of bound $0$ such that
	\begin{equation*}
	N^{\Delta_m,B_0} = W + P .
	\end{equation*}
	Therefore \cite[Thm.~3.8]{Paz:83} implies that $N^{\Delta_m,B_0}$ is relatively $\Delta_{\partial M}^g$-bounded of bound $0$. 
	Using the (uniform) ellipticity of $C$, there exists a constant $\Lambda > 0$ such that
	\begin{align*}
	\| \Delta_{\partial M}^g \varphi \|_{\rC(\partial M)}
	\leq \Lambda \cdot \| C \varphi \|_{\rC(\partial M)}
	\end{align*}
	for $\varphi \in D(C) = D(\Delta_{\partial M}^g)$. Hence $N^{\Delta_m,B_0}$ is relatively $C$-bounded of bound $0$. 
\end{proof}

Now the abstract results of \cite{BE:18} leads to the desired result.

\begin{thm}\label{LB gWBC}
	The operator $A^B$ with Wentzell boundary conditions associated to the Laplace-Beltrami operator $\Delta_m=\Delta_M^g$ generates a compact and analytic semigroup of angle $\nicefrac{\pi}{2}$ on $\rC(\overline{M})$.
\end{thm}
\begin{proof}
	We verify the assumptions of \cite[Thm.~4.3]{BE:18}.
	Remark that by \cite[Lem.~3.6]{Bin:18b} and \autoref{GWBC B A_0} above the Dirichlet operator $L_0 \in \mathcal{L}(\rC(\partial M),\rC(\overline{M}))$ exists and $B$ is relatively $A_0$-bounded of bound $0$.
	By multiplicative perturbation %(see \cite[Sect.~4]{Hol:92}) 
	we assume without loss of generality that $q~=~1$. 
	Now \cite[Thm.~1.1]{Bin:18a} 
	implies that $A_0$ is sectorial of angle $\nicefrac{\pi}{2}$ on $\rC(\overline{M})$ and has compact resolvent. Moreover by \cite[Cor.~3.6]{Bin:18a} the operator $C$ generates compact and analytic semigroup of angle $\nicefrac{\pi}{2}$ on $\rC(\partial M)$.
	Finally, the claim follows by \cite[Thm.~4.3]{BE:18}.
\end{proof}

\section{Elliptic operators with generalized Wentzell boundary conditions}

Consider a strictly elliptic differential operator $A_m \colon D(A_m) \subset \rC(\overline{M}) \to \rC(\overline{M})$ in divergence form on the boundary space. 
To this end, let 
\begin{align*}
a^k_j = a_j^k \in \rC^\infty(\overline{M}), \quad b_j \in \rC_c(\overline{M}), \quad c \in \rC(\overline{M}), \quad 1\leq j,k\leq n
\end{align*}
be real-valued functions, such that $a_j^k$ are elliptic,~i.e.
\begin{align*}
a_j^k(q) g^{jl}(q) X_k (q) X_l(q) > 0
\end{align*}
for all co-vectorfields $X_k,X_l$ on $\overline{M}$ with $(X_1(q),\dots,X_n(q)) \not = (0,\dots,0)$. Let $a = (a_j^k)_{j,k = 1,\dots,n}$ the $1$-$1$-tensorfield and $b = (b_j)_{j = 1,\dots,n}$. Then we define $A_m \colon D(A_m) \subset \rC(\overline{M}) \to \rC(\overline{M})$ by
\begin{align}
A_m f &:= \sqrt{|a|} \div_g \left( \frac{1}{\sqrt{|a|}} a \nabla_{M}^g f \right) + \langle b, \nabla_{M}^g f \rangle + c \cdot f, \label{Def: A_m M} \\
D(A_m) &:= \left\{ \varphi \in \bigcap_{p > 1} \rW^{2,p}_{\loc}(M) \cap \rC(\overline{M}) \colon A_m f \in \rC(\overline{M}) \right\}. \notag
\end{align}

We consider a $(2,0)$-tensorfield on $\overline{M}$ given by
\begin{align*}
\tilde{g}^{kl} = a^k_i g^{il} .
\end{align*}
Its inverse $\tilde{g}$ is a $(0,2)$-tensorfield on $\overline{M}$, which is a Riemannian metric since $a^k_j g^{jl}$ is strictly elliptic on $\overline{M}$. 
We denote $\overline{M}$ with the old metric by $\overline{M}^g$ and
with the new metric by $\overline{M}^{\tilde{g}}$ and remark that
$\overline{M}^{\tilde{g}}$ is a smooth, compact, orientable Riemannian manifold with smooth boundary $\partial M$. 
Since the differentiable structures of $\overline{M}^g$ and $\overline{M}^{\tilde{g}}$ coincide, the identity
\begin{equation*}
\Id \colon \overline{M}^g \longrightarrow \overline{M}^{\tilde{g}}
\end{equation*}
is a $\rC^\infty$-diffeomorphism. Hence,
the spaces
\begin{align*}
X := \rC(\overline{ M}) &:=\rC(\overline{M}^{\tilde{g}}) =\rC(\overline{M}^g)% = \tilde{X} 
\\
\text{and } \quad
\partial X := \rC(\partial M) &:=\rC(\partial M^{\tilde{g}}) =\rC(\partial M^g)% = \partial \tilde{X} 
\end{align*} 
coincide.
Moreover, \cite[Prop.~2.2]{Heb:00} implies that the following spaces coincide
\begin{align}
\rL^p(M) &:= \rL^p(M^{\tilde{g}}) = \rL^p(M^g), \notag \\
\rW^{k,p}(M) &:= \rW^{k,p}(M^{\tilde{g}}) = \rW^{k,p}(M^g), \notag \\
\rL^p_{loc}(M) &:= \rL^p_{loc}(M^{\tilde{g}}) = \rL^p_{loc}(M^g), \notag \\
\rW^{k,p}_{loc}(M) &:= \rW^{k,p}_{loc}(M^{\tilde{g}}) = \rW^{k,p}_{loc}(M^g), \label{sobolev spaces} \\
\rL^p(\partial M) &:= \rL^p(\partial M^{\tilde{g}}) = \rL^p(\partial M^g), \notag \\
\rW^{k,p}(\partial M) &:= \rW^{k,p}(\partial M^{\tilde{g}}) = \rW^{k,p}(\partial M^g), \notag \\
\rL^p_{loc}(\partial M) &:= \rL^p_{loc}(\partial M^{\tilde{g}}) = \rL^p_{loc}(\partial M^g), \notag \\
\rW^{k,p}_{loc}(\partial M) &:= \rW^{k,p}_{loc}(\partial M^{\tilde{g}}) = \rW^{k,p}_{loc}(\partial M^g) \notag  
\end{align} 
for all $p > 1$ and $k \in \N$.
Denote by $\hat{A}_m$ the maximal operator defined in \eqref{Def: A_m M} with $b_j = c = 0$ and by $\hat{C}$ the operator given in \eqref{Def: C} for $\beta_j = \gamma = 0$. Moreover, denote the corresponding feedback operator by $\hat{B}$.

\smallskip 

Next, we look at the operators $A_m$, $B_0$ and $C$ with respect to the new metric $\tilde{g}$. 

\begin{lem}\label{LB}
The operator $\hat{A}_m$ and the Laplace-Beltrami operator $\Delta_M^{\tilde{g}}$ coincide on $\rC(\overline{M})$.
\end{lem}
\begin{proof}
	Using local coordinates we obtain
	\begin{align*}
	\hat{A}_m f &= \frac{1}{\sqrt{|g|}} \sqrt{|a|} \partial_j \left(\sqrt{|g|} \frac{1}{\sqrt{|a|}} a_l^j g^{kl} \partial_k f\right) \\
	&= \frac{1}{\sqrt{|\tilde{g}|}} \partial_j \left(\sqrt{|\tilde{g}|} \tilde{g}^{kl} \partial_k f\right) = \Delta^{\tilde{g}}_m f
	\end{align*}
	for $f \in D(\hat{A}_m)= D(\Delta^{\tilde{g}}_m)$, since $|g| = |a| \cdot |\tilde{g}|$.
\end{proof}

Now we compare the maximal operators $A_m$ and $\hat{A}_m$.

\begin{lem}\label{Stoerung A}
The operators $A_m$ and $\hat{A}_m$ differ only by a relatively bounded perturbation of bound $0$. 
\end{lem}
\begin{proof}
Using \eqref{sobolev spaces} we define
\begin{align*}
P_1 f := b_l g^{kl} \partial_k f 
\end{align*}
for $f \in D(A_m) \cap D(\hat{A}_m)$.
Morreys embedding (cf. \cite[Chap.~V. and Rem.~5.5.2]{Ada:75}) implies
\begin{align}
\bigl[D(\hat{A}_m)\bigr] \stackrel{c}{\hookrightarrow} \rC^1({M}) \hookrightarrow \rC(M).
\label{Embeddings}
\end{align}
Since $b_l \in \rC_c(M)$, we obtain
\begin{align*}
	\| P_1f \|_{\rC(\overline{M})} &\leq \sup_{q \in \overline{M}} | b_l(q) g^{kl}(q) (\partial_k f)(q) |\\
	&= \sup_{q \in {M}} | b_l(q) g^{kl}(q) (\partial_k f)(q) | \\
	&\leq C \sum_{k = 1}^n \| \partial_k f \|_{\rC(M)}
\end{align*}
and therefore $P_1 \in \mathcal{L}(\rC^1({M}), \rC(\overline{M}))$. Hence $D(\hat{A}_m)=D(\tilde{A}_m)$. By \eqref{Embeddings} we conclude from Ehrling's Lemma
(see \cite[Thm.~6.99]{RR:04}) that
\begin{align*}
	\| P_1f \|_{\rC(\overline{M})} \leq C \| f \|_{\rC^1(M)} &\leq \varepsilon \| \hat{A}_m f \|_{\rC(\overline{M})} + \varepsilon \| f \|_{\rC(\overline{M})} + C(\varepsilon) \| f \|_{\rC(M)} \\
	&\leq \varepsilon \| \hat{A}_m f \|_{\rC(\overline{M})} + \tilde{C}(\varepsilon) \| f \|_{\rC(\overline{M})}
\end{align*}
for $f \in D(\hat{A}_m)$ and all $\varepsilon > 0$. Hence $P_1$ is relatively $A_m$-bounded of bound $0$. Finally remark that
\begin{align*}
	P_2 f := c \cdot f, \quad D(P_2) := \rC(\overline{M})
\end{align*}
is bounded and that
\begin{align*}
	\tilde{A}_m f = \hat{A}_m f + P_1 f + P_2 f
\end{align*}
for $f \in D(\hat{A}_m)$.
\end{proof}

\begin{lem}\label{Stoerung B}
The operators $B_0$ and the negative conormal derivative $-\frac{\partial^{\tilde{g}}}{\partial \nu}$ coincide.
\end{lem}
\begin{proof}
Since the Sobolev spaces coincide, we compute in local coordinates
\begin{align*}
- \frac{\partial^{\tilde{g}}}{\partial \nu} f 
&= - g_{ij} g^{jl} a_l^k \partial_k f g^{im} \nu_m  \\
&= - g_{ij} \tilde{g}^{jl} \partial_k f g^{im} \nu_m \\
&= - \tilde{g}_{ij} \tilde{g}^{jl} \partial_k f \tilde{g}^{im} \nu_m \\ 
&= B_0 f
\end{align*}
for $f \in D(B)=D(\frac{\partial^{\tilde{g}}}{\partial \nu})$.
\end{proof} 

%By \autoref{LB} and \autoref{Stoerung B} the maximal operator $\hat{A}_m$ is the Laplace-Beltrami operator $\Delta_M^{\tilde{g}}$ with respect to $\tilde{g}$ and $B_0$ is the (negative) normal derivate $\frac{\partial}{\partial \nu^{\tilde{g}}}$ with respect to $\tilde{g}$. So, how does $C$ look like with respect to the new Riemannian metric $\tilde{g}$?

%\smallskip 

Define $\tilde{C} \colon D(\tilde{C}) \subset \rC(\partial M) \to \rC(\partial M)$ by
\begin{align*}
\tilde{C} \varphi := \sqrt{|\tilde{\alpha}|} \div_{\tilde{g}} \left( \frac{1}{\sqrt{|\tilde{\alpha}|}} \tilde{\alpha} \nabla_{\partial M}^{\tilde{g}} \varphi \right), \quad 
D(C) := \{ \varphi \in \rW^{2,p}(\partial M) \colon C \varphi \in \rC(\partial M) \}, 
\end{align*}
where $\tilde{\alpha}(q) := a(q)^{-1} \cdot \alpha(q)$.

\begin{lem}\label{C tilde g}
	The operators $\hat{C}$ and $\tilde{C}$ coincide on $\rC(\partial M)$.
\end{lem}
\begin{proof}
	An easy calculation shows
	\begin{align*}
	\frac{|\tilde{g}|}{|\tilde{\alpha}|} &= \frac{|g|}{|\alpha|}, \\
	\tilde{\alpha}^k_l \tilde{g}^{lj} &= \alpha^k_l g^{lj} .
	\end{align*}
	Hence we obtain in local coordinates
	\begin{align*}
	\tilde{C} \varphi &= \sqrt{\frac{|\tilde{\alpha}|}{|\tilde{g}|}} \partial_k \left( \sqrt{\frac{|\tilde{g}|}{|\tilde{\alpha}|}} \tilde{\alpha}^k_l \tilde{g}^{li} \partial_i \varphi \right) \\
	&= \sqrt{\frac{|{\alpha}|}{|{g}|}} \partial_k \left( \sqrt{\frac{|{g}|}{|{\alpha}|}} {\alpha}^k_l {g}^{li} \partial_i \varphi \right) \\
	&= \sqrt{|\alpha|} \div_g \left( \frac{1}{|\alpha|} \alpha \nabla^j \varphi	\right) =\hat{C} \varphi 
	\end{align*}
	for $\varphi \in D(\hat{C})=D(\tilde{C})$.
\end{proof}

Next we compare the operators $C$ and $\hat{C}$.

\begin{lem}\label{Stoerung C}
	The operators $C$ and $\hat{C}$ differ only by a relatively bounded perturbation of bound $0$.
\end{lem}
\begin{proof}
	Denote by
	\begin{equation*}
	P \varphi := \langle \beta, \nabla_{\partial M}^g \rangle + \gamma \cdot \varphi \text{ for } f \in D(P) := \rC^1(\partial M) 
	\end{equation*}
	and note that $P \in \mathcal{L}(\rC^1({\partial M}), \rC(\partial M))$. The Sobolev embeddings and the closed graph theorem imply
	\begin{align*}
	[D(C)] \stackrel{c}{\hookrightarrow} \rC^1(\partial M) \hookrightarrow \rC(\partial M).
	\end{align*}
	Finally, the claim follows by Ehrling's Lemma (cf. \cite[Thm.~6.99]{RR:04}).
\end{proof}

Now we are prepared to prove our main theorem.

\begin{thm}\label{mainthm}
	The operator $A^B$ with Wentzell boundary conditions generates a compact and analytic semigroup of angle $\nicefrac{\pi}{2}$ on $\rC(\overline{M})$.
\end{thm}
\begin{proof}
	Since $\tilde{C}$ is a strictly elliptic differential operator in divergence form on $\rC(\partial M)$ we obtain by \autoref{LB gWBC} that the Laplace-Beltrami operator with Wentzell boundary conditions given by 
	\begin{equation*}
	 (\Delta_{M}^{\tilde{g}} f)|_{\partial M} =	q \cdot \tilde{C} f|_{\partial M} -\eta \frac{\partial^{\tilde{g}}}{\partial \nu}f  
	\end{equation*}
	generates a compact and analytic semigroup of angle $\nicefrac{\pi}{2}$ on $\rC(\overline{M})$. 
	Now \autoref{LB}, \autoref{Stoerung B} and \autoref{C tilde g} imply that the operator $\hat{A}^{\hat{B}}$ generates a compact and analytic semigroup of angle $\nicefrac{\pi}{2}$ on $\rC(\overline{M})$.
Note that $A_m$ and $\hat{A}_m$ differ only by a relatively $A_m$-bounded perturbation of bound $0$ by \autoref{Stoerung A}.
By \autoref{Stoerung C} one obtains that the perturbation on the boundary is relatively $\hat{C}$-bounded. %Note that $\hat{C}=N^{A_m,\hat{C}L}$.
Now the claim follows from \cite[Thm.~4.2]{BE:18}.
%\footnote{	Since we use the symbols $A,B,C$ twice (once for abstract operators and once for concrete differential operators), here a problem of abusing notation appears. The operator $C$ of \autoref{Stoerung} is \textbf{not} the operator $C$ from \eqref{Def: C}. Let me explain how \autoref{Stoerung} should applied here. We denote by $T^{\text{abstr}}$ the abstract operators from \textbf{Section 2}. The operator $P^{\text{abstr}}$	is the operator defined in the proof of \autoref{Stoerung D-N L^2}. The operator $N^{A_{m}^{\text{abst}},B_{0}^{\text{abst}}} = N^{A_m,\hat{B}}$, whereas $C^{\text{abst}}$ is the perturbation defined in the proof of \autoref{Stoerung C}.}.
\end{proof}

\begin{rem}
	\autoref{mainthm} generalizes the main theorem in \cite{GGP:17} for the case $p = \infty$. 
\end{rem}

\begin{cor}
	The initial-value boundary problem
	\begin{alignat*}{4}
	\left\{
	\begin{array}{lll}
	\frac{d}{dt} u(t,q) &= A_m u(t,q), \quad &t \geq 0, \ q \in \overline{M}, \\
	\frac{d}{dt} \varphi(t,q) &= B u(t,q), &t \geq 0, \ q \in \partial{M}, \\
	u(t,x) &= \varphi(t,x), &t \geq 0, \ x \in \partial{M}, \\
	u(0,q) &= u_0(q) &q \in \overline{M},
	\end{array}
	\right. 
	\end{alignat*}
	on $\rC(\overline{M})$ is well-posed. Moreover the solution $\begin{psmallmatrix}
	u(t) \\ \varphi(t)
	\end{psmallmatrix} \in \rC^\infty(M) \times \rC^\infty(\partial M)$ for $t > 0$ depends analytically on the initial value $\begin{psmallmatrix}
	u_0 \\ u_0|_{\partial M}
	\end{psmallmatrix}$ and is governed by a compact and analytic semigroup, which can be extended to a right half plane.  
\end{cor} 

\nocite{EN:00}
\bigskip 
%\printbibliography%[maxnames = 10]

%\bibliographystyle{alphadin}
%\bibliography{book}

\begin{thebibliography}{FGG{\etalchar{+}}10}
	\providecommand{\url}[1]{\texttt{#1}}
	\providecommand{\urlprefix}{}

\bibitem[Ada75]{Ada:75}
R.~A.~Adams.
\newblock \emph{Sobolev Spaces},
\newblock Academic Press, 1975
\\

\bibitem[BE19]{BE:18}
T.~Binz and K.-J.~Engel.
\newblock \emph{Operators with Wentzell boundary conditions and the
Dirichlet-to-Neumann operator},
Math. Nachr. \textbf{292} (2019), 733 -- 746
\\

\bibitem[Bin18a]{Bin:18b}
T.~Binz.
\newblock \emph{Dirichlet-to-Neumann operators on manifolds}
\newblock (preprint, 2018)
\\

\bibitem[Bin18b]{Bin:18a}
T.~Binz.
\newblock \emph{Strictly elliptic Operators with {D}irichlet boundary
	conditions on spaces of continuous functions on manifolds}
\newblock (preprint, 2018)
\\

\bibitem[EF05]{EF:05}
K.-J.~Engel and G.~Fragnelli.
\newblock \emph{Analyticity of semigroups generated by operators with generalized
{W}entzell boundary conditions},
Adv. Differential Equations \textbf{10} (2005), 1301--1320
\\

\bibitem[EN00]{EN:00}
K.-J.~Engel and R.~Nagel.
\newblock \emph{One-{P}arameter {S}emigroups for {L}inear {E}volution
	{E}quations}.
\newblock Springer, 2000
\\

\bibitem[Eng03]{Eng:03}
K.-J.~Engel.
\newblock \emph{The {L}aplacian on {$C(\overline\Omega)$} with generalized {W}entzell
boundary conditions},
Arch.~Math.~\textbf{81}~(2003), 548--558
\\

\bibitem[FGG{\etalchar{+}}10]{FGGR:10}
A.~Favini, G.~Goldstein, J.~A.~Goldstein,
E.~Obrecht and S.~Romanelli.
\newblock \emph{Elliptic operators with general Wentzell boundary conditions,
analytic semigroups and the angle concavity theorem},
Math. Nachr. \textbf{283} (2010), 504 – 521
\\

\bibitem[GGP17]{GGP:17}
J.~A.~Goldstein, G.~Goldstein, M.~Pierre.
\newblock \emph{The Agmon-Douglis-Nirenberg problem in the context of dynamic
	boundary conditions}
\newblock (preprint 2017)
\\

\bibitem[Heb00]{Heb:00}
E.~Hebey.
\newblock \emph{Nonlinear Analysis on Manifolds: Sobolev Spaces and
	Inequalities}.
\newblock Amer. Math. Soc., 2000
\\

\bibitem[Paz83]{Paz:83}
A.~Pazy.
\newblock \emph{Semigroups of {L}inear {O}perators and {A}pplications to
	{P}artial {D}ifferential {E}quations}.
\newblock Springer, 1983
\\

\bibitem[RR04]{RR:04}
M.~Renardy and R.~C.~Rogers.
\newblock \emph{An Introduction to Partial Differential Equations}.
\newblock Springer, 2004
\\

\bibitem[Tay96]{Tay:96}
M.~E.~Taylor.
\newblock \emph{{Partial} {Differential} {Equations} {II}.}
\newblock Springer, 1996

	
\end{thebibliography}

\newcommand{\etalchar}[1]{$^{#1}$}

\
\\[1.5em]
\emph{Tim Binz}, University of Tübingen, Department of Mathematics, Auf der Morgenstelle 10, D-72076 Tübingen, Germany,
\texttt{tibi@fa.uni-tuebingen.de}

\end{document}